\documentclass[12pt, a4paper, reqno]{amsart}
\usepackage{amsmath}
\usepackage{amsfonts}
\usepackage{amssymb}
\usepackage{color}
\usepackage{comment,graphicx}
\usepackage[T1]{fontenc}
\usepackage{url}
\usepackage{ae}

\def\phi{\varphi}
\def\rho{\varrho}
\def\epsilon{\varepsilon}
\numberwithin{equation}{section}
\theoremstyle{plain}
\newtheorem{theorem}[equation]{Theorem}
\newtheorem{lemma}[equation]{Lemma}
\newtheorem{proposition}[equation]{Proposition}

\theoremstyle{definition}
\newtheorem{definition}[equation]{Definition}

\theoremstyle{remark}
\newtheorem{remark}[equation]{Remark}

\renewcommand{\leq}{\leqslant}
\renewcommand{\geq}{\geqslant}

\begin{document}
\author[D. Drihem]{Douadi Drihem \ }
\date{\today }
\title[$f\mapsto |f|^{\mu },\mu >0$ on Besov spaces of power weights]{Powers
functions in Besov spaces of power weights. Necessary conditions}
\maketitle

\begin{abstract}
The aim of this paper is to present some necessary conditions for the
boundedness of the mapping $f\mapsto |f|^{\mu },\mu >0$ on Besov spaces
equipped with power weights.$\newline
$


\end{abstract}

\section{Introduction}

Let $G:\mathbb{R}\rightarrow \mathbb{R}$ be a continuous function. The
associate composition operator $T_{G}$ is defined by $f\mapsto G(f)$. The
study of the action of this operator on functional spaces comes not only
from theoretical purposes, but also from applications to partial
differential equations. Many nonlinear equations are stated in a form where
the nonlinear part is given by a composition operator, for example the
nonlinear equations 
\begin{align*}
\left\{ 
\begin{array}{ccc}
\partial _{t}f(t,x)-\Delta f(t,x)= T_{G}(f(t,x)), & \quad (t,x)\in \mathbb{\
R}^{+}\mathbb{\times R}^{n}. &  \\ 
\text{ \ \ \ \ \ \ \ \ }f(0,x)=f_{0}(x). &  & 
\end{array}%
\right.
\end{align*}%
To study this equation in functional spaces such as Sobolev spaces we need
to estimate the nonlinear term $T_{G}$ in such spaces, see for example \cite%
{F98}.

Another motivation to study the composition operators in function spaces can
be found in \cite{CFZ11} and the references therein.

Concerning the composition operators in Sobolev, Besov and Triebel-Lizorkin
spaces we refer to \cite{Bo91}, \cite{BK}, \cite{BCS06}, \cite{BMS10} and 
\cite{Runst-Sickel1996}. Recently the author in \cite{Dr21} gave necessary
and sufficient conditions on $G$ such that 
\begin{equation}
T_{G}(W_{p}^{m}(\mathbb{R}^{n},|\cdot |^{\alpha }))\subset W_{p}^{m}(\mathbb{%
\ R}^{n},|\cdot |^{\alpha }),  \label{Dah1}
\end{equation}%
with some suitable assumptions on $m,p$ and $\alpha $. More precisely, he
proved the following result. Let $m=2,3,...$ and let $1<p<\infty ,0\leq
\alpha <n(p-1)$. Assume that $m>\frac{n+\alpha }{p}$. Then the composition
operator $T_{G}$ satisfies \eqref{Dah1} if and only if $G$ satisfies the
following conditions: 
\begin{equation*}
G(0)=0\quad \text{and}\quad G^{(m)}\in \mathbb{L}_{\mathrm{loc}}^{p}(\mathbb{%
\ R}).
\end{equation*}%
Some sufficient conditions on $G$ which ensure 
\begin{equation*}
T_{G}(F_{p,q}^{s}(\mathbb{R}^{n},|\cdot |^{\alpha }))\subset F_{p,q}^{s}(%
\mathbb{R}^{n},|\cdot |^{\alpha })
\end{equation*}%
with application to semilinear parabolic equations are given in\ \cite%
{DrBanach}. Here $F_{p,q}^{s}(\mathbb{R}^{n},|\cdot |^{\alpha })$ denotes
the Triebel-Lizorkin space.

The main purpose of this paper is to present necessary conditions about
mapping properties of power type like%
\begin{equation}
G_{\mu }:f\longmapsto |f|^{\mu },\quad \mu >0,  \label{Problem1}
\end{equation}%
in the framework of Besov spaces equipped with power weights.

For $\mu =1$, Bourdaud and Meyer \cite{Bourdaud-Meyer91} proved that $%
G_{1}(f)\in B_{p,q}^{s}$ for all $f\in B_{p,q}^{s},1\leq p,q\leq \infty $
and $0<s<1+\frac{1}{p}$. Independently, the same result was obtained by
Oswald \cite{Os91}.

Since the function $G_{\mu }$, $\mu <1$ is not locally Lipschitz continuous,
the mapping $G_{\mu },\mu <1$, cannot be a mapping in the same Besov space.
In \cite[5.4.4/1]{Runst-Sickel1996} we obtain that $G_{1}(f)\in B_{\frac{p}{%
\mu },\mu }^{s\mu }$ for all $f\in B_{p,q}^{s},1\leq p,q\leq \infty $ and $%
0<s<1+\frac{1}{p}.$

Kateb \cite{K03} gave a complete treatment of \eqref{assumptions} with $\mu
>1$. We mention that the action of these type of operators in function
spaces play an important role in mathematical analysis, see \cite{CW90}
where the mapping $f\longmapsto |f|^{\mu }f,\mu >0$ plays a crucial role.

The paper is organized as follows. First we give some preliminaries where we
fix some notation and recall some basic facts on Besov spaces equipped with
power weights. Also, we give some key technical results needed in the proof
of the main statement. The main statement is formulated in the
\textquotedblleft Proof of the results\textquotedblright\ section. More
precisely, the following results are true.\textbf{\ }Let%
\begin{equation}
1\leq p<\infty ,\quad \alpha >-n\quad \text{and}\quad s>\max \Big(0,\frac{%
\alpha +n}{p}-n\Big).  \label{assumptions}
\end{equation}

\begin{theorem}
\label{Result1}Let $1\leq p,q<\infty $ and \eqref{assumptions} with $n=1$.
Assume that $\mu $ is not an even integer. In order to have embeddings like%
\begin{equation*}
G_{\mu }(B_{p,q}^{s}(\mathbb{R},|\cdot |^{\alpha }))\subset B_{p,q}^{s}(%
\mathbb{R},|\cdot |^{\alpha }),
\end{equation*}%
the condition $s<\mu +\frac{1+\alpha }{p}$\ becomes necessary$.$
\end{theorem}

\begin{theorem}
\label{Result2}Let $1\leq q\leq \infty $ and \eqref{assumptions}.$\newline
\mathrm{(i)}$ Let $\alpha \geq 0$. Assume that $\mu >1,s<\frac{n+\alpha }{p}$
or $s=\frac{n+\alpha }{p}$\ and\ $q>1$. To study\ \eqref{Problem1}\ in the
framework of Besov spaces of power weights\ one needs in addition $f\in
L^{\infty }(\mathbb{R}^{n}).\newline
\mathrm{(ii)}$ Assume that $\mu <1\ $and $s>\frac{n+\alpha }{p}$. The
condition $f\notin L^{\infty }(\mathbb{R}^{n})$\ is necessary for the
solvability of \eqref{Problem1}\ in the framework of Besov spaces of power
weights$.$
\end{theorem}

\begin{theorem}
\label{Result3 copy(1)}Let $1\leq q\leq \infty $ and \eqref{assumptions}
with $n=1$. Assume that $\mu <1,\frac{\alpha }{p\mu }\leq s\leq \frac{%
1+\alpha }{p}\ $and $0<\alpha <p-1$. The condition $f\notin L^{\infty }(%
\mathbb{R})$\ is necessary for the solvability of \eqref{Problem1}\ in the
framework of Besov spaces of power weights$.$
\end{theorem}

\section{Function spaces}

In this section we present the Fourier analytical definition of Besov spaces
equipped with power weights, and their basic properties. First, we need some
notation. As usual, $\mathbb{R}^{n}$ denotes the $n$-dimensional real
Euclidean space, $\mathbb{N} $ the collection of all natural numbers and $%
\mathbb{N}_{0}=\mathbb{N}\cup \{0\}$. The letter $\mathbb{Z}$ stands for the
set of all integer numbers.

Let $\alpha \in \mathbb{R}$ and $0<p<\infty $. The weighted Lebesgue space $%
L^{p}(\mathbb{R}^{n},|\cdot |^{\alpha })$ contains all measurable functions
such that 
\begin{equation*}
\big\|f\big\|_{L^{p}(\mathbb{R}^{n},|\cdot |^{\alpha })}=\Big(\int_{\mathbb{R%
}^{n}}\left\vert f(x)\right\vert ^{p}|x|^{\alpha }dx\Big)^{1/p}<\infty .
\end{equation*}

If $\alpha =0$, then we put $L^{p}(\mathbb{R}^{n},|\cdot |^{0})=L^{p}$ and $%
\big\|f\big\|_{L^{p}(\mathbb{R}^{n},|\cdot |^{0})}=\big\|f\big\|_{p}$.

The symbol $\mathcal{S}(\mathbb{R}^{n})$ is used in place of the set of all
Schwartz functions defined on $\mathbb{R}^{n}$ and we denote by $\mathcal{S}%
^{\prime }(\mathbb{R}^{n})$ the dual space of all tempered distributions on $%
\mathbb{R}^{n}$. We define the Fourier transform of a function $f\in 
\mathcal{S}(\mathbb{R}^{n})$ by%
\begin{equation*}
\mathcal{F(}f)(\xi )=\left( 2\pi \right) ^{-n/2}\int_{\mathbb{R}%
^{n}}e^{-ix\cdot \xi }f(x)dx,\quad \xi \in \mathbb{R}^{n}.
\end{equation*}%
Its inverse is denoted by $\mathcal{F}^{-1}f$. Both $\mathcal{F}$ and $%
\mathcal{F}^{-1}$ are extended to the dual Schwartz space $\mathcal{S}%
^{\prime }(\mathbb{R}^{n})$ in the usual way.

Next, we define the function spaces $B_{p,q}^{s}(\mathbb{R}^{n},|\cdot
|^{\alpha })$. We first need the concept of a smooth dyadic resolution of
unity. Let $\psi $ be a function in $\mathcal{S}(\mathbb{R}^{n})$ satisfying 
$\psi (x)=1$ for $\lvert x\rvert \leq 1$ and $\psi (x)=0$ for $\lvert
x\rvert \geq \frac{3}{2}$. We put $\varphi _{0}(x)=\psi (x),\,\varphi
_{1}(x)=\psi (\frac{x}{2})-\psi (x)$ and $\varphi _{j}(x)=\varphi
_{1}(2^{-j+1}x)\ $for$\ j=2,3,....$ Then we have supp$\varphi _{j}\subset
\{x\in {\mathbb{R}^{n}}:2^{j-1}\leq \lvert x\rvert \leq 3\cdot 2^{j-1}\}\ $%
and $\sum_{j=0}^{\infty }\varphi _{j}(x)=1$ for all $x\in {\mathbb{R}^{n}}$.
The system of functions $\{\varphi _{j}\}_{j\in \mathbb{N}_{0}}$ is called a
smooth dyadic resolution of unity. Thus we obtain the Littlewood-Paley
decomposition 
\begin{equation*}
f=\sum_{j=0}^{\infty }\mathcal{F}^{-1}\varphi _{j}\ast f
\end{equation*}%
for all $f\in \mathcal{S}^{\prime }(\mathbb{R}^{n})$ (convergence in $%
\mathcal{S}^{\prime }(\mathbb{R}^{n})$).

We are now in a position to state the definition of Besov spaces equipped
with power weights.

\begin{definition}
\label{def-herz-Besov}Let $\alpha ,s\in \mathbb{R}$, $0<p<\infty $\ and $%
0<q\leq \infty $. The Besov space $B_{p,q}^{s}(\mathbb{R}^{n},|\cdot
|^{\alpha })$ is the collection of all $f\in \mathcal{S}^{\prime }(\mathbb{R}%
^{n})$ such that 
\begin{equation*}
\big\Vert f\big\Vert_{B_{p,q}^{s}(\mathbb{R}^{n},|\cdot |^{\alpha })}=\Big(%
\sum_{j=0}^{\infty }2^{jsq}\big\Vert\mathcal{F}^{-1}\varphi _{j}\ast f%
\big\Vert_{L^{p}(\mathbb{R}^{n},|\cdot |^{\alpha })}^{q}\Big)^{1/q}<\infty ,
\end{equation*}%
with the obvious modification if $q=\infty $.
\end{definition}

\begin{remark}
Let\textrm{\ }$s\in \mathbb{R},0<p<\infty ,0<q\leq \infty $ and $\alpha >-n$%
. The spaces\textrm{\ }$B_{p,q}^{s}(\mathbb{R}^{n},|\cdot |^{\alpha })$ are
independent of the particular choice of the smooth dyadic resolution of unity%
\textrm{\ }$\{\varphi _{j}\}_{j\in \mathbb{N}_{0}}$ (in the sense of$\mathrm{%
\ }$equivalent quasi-norms). In particular $B_{p,q}^{s}(\mathbb{R}%
^{n},|\cdot |^{\alpha })$\textrm{\ }are quasi-Banach spaces and if\textrm{\ }%
$p,q\geq 1$, then $B_{p,q}^{s}(\mathbb{R}^{n},|\cdot |^{\alpha })$\ are
Banach spaces, see \cite{Bui82}.
\end{remark}

Now we give the definition of the spaces $B_{p,q}^{s}$.

\begin{definition}
Let $s\in \mathbb{R}$\ and $0<p,q\leq \infty $. The Besov space $B_{p,q}^{s}$
is the collection of all $f\in \mathcal{S}^{\prime }(\mathbb{R}^{n})$\ such
that 
\begin{equation*}
\big\|f\big\|_{B_{p,q}^{s}}=\Big(\sum\limits_{j=0}^{\infty }2^{jsq}\big\|%
\mathcal{F}^{-1}\varphi _{j}\ast f\big\|_{p}^{q}\Big)^{1/q}<\infty ,
\end{equation*}%
with the obvious modification if $q=\infty $.
\end{definition}

The theory of the spaces $B_{p,q}^{s}$ has been developed in detail in \cite%
{Triebel1983} but has a longer history already including many contributors;
we do not want to discuss this here. Clearly, for $s\in \mathbb{R}%
,0<p<\infty $ and $0<q\leq \infty ,$%
\begin{equation*}
B_{p,q}^{s}(\mathbb{R}^{n},|\cdot |^{\alpha })=B_{p,q}^{s}\quad \text{if}%
\quad \alpha =0.
\end{equation*}

To prove our results of this paper we need some embeddings. The following
statement holds by \cite[Theorem 5.9]{Drihem2013a} and \cite{MM12}.

\begin{theorem}
\label{embeddings5}\textit{Let }$\alpha _{1},\alpha _{2},s_{1},s_{2}\in 
\mathbb{R},0<p_{1},p_{2}<\infty ,0<q\leq \infty ,\alpha _{1}>-n\ $\textit{%
and }$\alpha _{2}>-n$. \textit{We suppose that }%
\begin{equation*}
s_{1}-\frac{n+\alpha _{1}}{p_{1}}\leq s_{2}-\frac{n+\alpha _{2}}{p_{2}}.
\end{equation*}

\textit{Let }$0<p_{2}\leq p_{1}<\infty $ and $\frac{\alpha _{2}}{p_{2}}\geq 
\frac{\alpha _{1}}{p_{1}}$. Then%
\begin{equation*}
B_{p_{2},q}^{s_{2}}(\mathbb{R}^{n},|\cdot |^{\alpha _{2}})\hookrightarrow
B_{p_{1},q}^{s_{1}}(\mathbb{R}^{n},|\cdot |^{\alpha _{1}}).
\end{equation*}
\end{theorem}

We present some characterization of the above spaces in terms of ball mains
of differences. Let $f$ be an arbitrary function on $\mathbb{R}^{n}$ and $%
x,h\in \mathbb{R}^{n}$. Then%
\begin{equation*}
\Delta _{h}f(x)=f(x+h)-f(x),\quad \Delta _{h}^{M+1}f(x)=\Delta _{h}(\Delta
_{h}^{M}f)(x),\quad M\in \mathbb{N}.
\end{equation*}%
These are the well-known differences of functions which play an important
role in the theory of function spaces. Using mathematical induction one can
show the explicit formula%
\begin{equation*}
\Delta _{h}^{M}f(x)=\sum_{j=0}^{M}\left( -1\right) ^{j}C_{M}^{j}f(x+(M-j)h),
\end{equation*}%
where $C_{M}^{j}$ are the binomial coefficients. By ball means of
differences we mean the quantity%
\begin{equation*}
d_{t}^{M}f(x)=t^{-n}\int_{|h|<t}\left\vert \Delta _{h}^{M}f(x)\right\vert
dh=\int_{B}\left\vert \Delta _{th}^{M}f(x)\right\vert dh,
\end{equation*}%
where $B=\{y\in \mathbb{R}^{n}:|h|\leq 1\}$ is the unit ball of $\mathbb{R}%
^{n}$, $t>0$ is a real number and $M$ is a natural number. Let%
\begin{equation}
\big\|f\big\|_{B_{p,q}^{s}(\mathbb{R}^{n},|\cdot |^{\alpha })}^{\ast }=\big\|%
f\big\|_{L^{p}(\mathbb{R}^{n},|\cdot |^{\alpha })}+\Big(\int_{0}^{1}t^{-sq}%
\big\|d_{t}^{M}f\big\|_{L^{p}(\mathbb{R}^{n},|\cdot |^{\alpha })}^{q}\frac{dt%
}{t}\Big)^{\frac{1}{q}}  \label{norm1}
\end{equation}%
The following theorems play a central role in our paper, see \cite{Drihem20}.

\begin{theorem}
\label{means-diff-cha1}\textit{Let }$0<p<\infty ,0<q\leq \infty ,\alpha >-n$%
, $\alpha _{0}=n-\frac{n}{p}\ $and $M\in \mathbb{N}$. Assume that%
\begin{equation*}
\max \Big(\sigma _{p},\frac{\alpha }{p}-\alpha _{0}\Big)<s<M,\quad \sigma
_{p}=\max \Big(0,\frac{n}{p}-n\Big).
\end{equation*}%
Then $\big\|\cdot \big\|_{B_{p,q}^{s}(\mathbb{R}^{n},|\cdot |^{\alpha
})}^{\ast }$ is equivalent quasi-norm on $B_{p,q}^{s}(\mathbb{R}^{n},|\cdot
|^{\alpha })$.
\end{theorem}

\begin{remark}
The integral $\int_{0}^{1}\cdot \cdot \cdot \frac{dt}{t}$ in \eqref{norm1}
can be replaced by $\int_{0}^{\varepsilon }\cdot \cdot \cdot \frac{dt}{t}$, $%
0<\varepsilon <1$ in the sense of equivalent quasi-norms.
\end{remark}

\section{Proof of the main results}

In this section we prove our statements. Our proofs use partially some
decomposition techniques already used in \cite{Bourdaud04}\ and \cite%
{Runst-Sickel1996} where the Besov case was studied. The first result based
on the following Lemma, see \cite{Dr-Nim}.

\begin{lemma}
\label{Si-funct1}Let $0<p<\infty ,0<q\leq \infty ,\alpha >-n$, $\alpha
_{0}=n-\frac{n}{p}\ $and $s>\max (\sigma _{p},\frac{\alpha }{p}-\alpha _{0})$%
. We put%
\begin{equation}
f_{\mu ,\delta }(x)=\theta (x)|x|^{\mu }(-\log |x|)^{-\delta },
\label{f-alpha-delta}
\end{equation}%
where $\mu ^{2}+\delta ^{2}>0,\delta \geq 0$, $\mu \neq 0$ and $\theta $ is
a smooth cut-off function with $\mathrm{supp}\theta \subset \{x:|x|\leq
\vartheta \}$, $\vartheta >0$ sufficiently small.$\newline
\mathrm{(i)}$ Let $\delta >0$. Then $f_{\mu ,\delta }\in B_{p,q}^{s}(\mathbb{%
R}^{n},|\cdot |^{\alpha })$ if and only if $s<\frac{n+\alpha }{p}+\mu $ or $%
s=\frac{n+\alpha }{p}+\mu $ and $q\delta >1$.$\newline
\mathrm{(ii)}$ We have $f_{\mu ,0}\in B_{p,q}^{s}(\mathbb{R}^{n},|\cdot
|^{\alpha })$ if and only if $s<\frac{n+\alpha }{p}+\mu $ or $s=\frac{%
n+\alpha }{p}+\mu $ and $q=\infty $.
\end{lemma}

For any $u>0$, $k\in \mathbb{Z}$ we set $C(u)=\{x\in \mathbb{R}^{n}:\frac{u}{%
2}\leq \left\vert x\right\vert <u\}$ and $C_{k}=C(2^{k})$. For $x\in \mathbb{%
R}^{n}$ and $r>0$ we denote by $B(x,r)$ the open ball in $\mathbb{R}^{n}$
with center $x$ and radius $r$. Let $\chi _{k}$, for $k\in \mathbb{Z}$,
denote the characteristic function of the set $C_{k}$. Before state the
proof of this lemma we mention that%
\begin{equation}
\big\|g\big\|_{L^{p}(\mathbb{R}^{n},|\cdot |^{\alpha })}\approx \Big(%
\sum\limits_{k=-\infty }^{\infty }2^{k\alpha }\big\|g\chi _{k}\big\|_{p}^{p}%
\Big)^{\frac{1}{p}}  \label{discreteversion}
\end{equation}%
for any $g\in L^{p}(\mathbb{R}^{n},|\cdot |^{\alpha })$, $\alpha \in \mathbb{%
R}$ and $0<p<\infty $. This discrete version play a central role in our
paper.

Under the hypotheses of Lemma \ref{Si-funct1}\textbf{\ }we have 
\begin{equation*}
B_{p,q}^{s}(\mathbb{R}^{n},|\cdot |^{\alpha })\hookrightarrow L_{loc}^{1}
\end{equation*}%
see \cite[Theorem 4]{Drihem20}. Then the restriction $s>\max (\sigma _{p},%
\frac{\alpha }{p}-\alpha _{0})$ is natural. Here we interpret $L_{loc}^{1}$
as the set of regular distributions. If $\alpha =0$, then Lemma \ref%
{Si-funct1} reduces to the result of \cite[Lemma 2.3.1/1]{Runst-Sickel1996}.
We do not deal with the case $\mu =0$, however for the classical Besov
spaces is given in \cite[Lemma 2.3.1/2]{Runst-Sickel1996}.

\begin{proof}[\textbf{Proof of Theorem \protect\ref{Result1}}]
Let $\theta $ be as in Lemma \ref{Si-funct1} and $f(x)=x_{1}\theta
(|x|),x\in \mathbb{R}^{n}$ which belong to $\mathcal{D}(\mathbb{R}%
^{n})\subset B_{p,q}^{\frac{n+\alpha }{p}+\mu }(\mathbb{R}^{n},|\cdot
|^{\alpha })$. Proceed in two cases.

\textbf{Case 1. }$p>1$. Let $\alpha _{1}<\alpha $ be such that $0<\frac{%
n+\alpha _{1}}{p}+\mu <1$. We claim that $G_{\mu }(f)\notin B_{p,q}^{\frac{%
n+\alpha _{1}}{p}+\mu }(\mathbb{R}^{n},|\cdot |^{\alpha _{1}})$, which
implies that $G_{\mu }(f)$ does not belong to $B_{p,q}^{\frac{n+\alpha }{p}%
+\mu }(\mathbb{R}^{n},|\cdot |^{\alpha })$, see Theorem \ref{embeddings5}.
Let us prove our claim. By $D(t)$ we denote the set 
\begin{equation*}
\Big\{x\in \mathbb{R}^{n}:\left\vert x\right\vert <\frac{t}{2H}\Big\},\quad
H>0,
\end{equation*}%
with $0<t<\varepsilon $ and $\varepsilon $ sufficiently small . Let $j\in 
\mathbb{Z}$ be such that $2^{j-1}\leq \frac{t}{2H}<2^{j}$. It is easily seen
that%
\begin{align*}
\sum\limits_{k=-\infty }^{\infty }2^{k\alpha }\big\|(d_{t}^{1}(G_{\mu
}(f)))\chi _{C_{k}}\big\|_{p}^{p}\geq & 2^{(i-1)\alpha }\big\|%
(d_{t}^{1}(G_{\mu }(f)))\chi _{D(t)\cap C_{j-1}}\big\|_{p}^{p} \\
\geq & ct^{\alpha }\big\|(d_{t}^{1}(G_{\mu }(f)))\chi _{D(t)\cap C_{j-1}}%
\big\|_{p}^{p},
\end{align*}%
where $c$ is independent of $t$ and $j$. Let $x\in D(t)\cap C_{j-1}$ and 
\begin{equation*}
C=\Big\{h:|h_{i}|\leq \frac{t}{\sqrt{n}},\quad \frac{3t}{H}\leq |h_{1}|\leq 
\frac{t}{\sqrt{n}},i=2,...,n\Big\}.
\end{equation*}%
By the inequality (27) in \cite[2.3.1, p. 45]{Runst-Sickel1996}, we obtain 
\begin{equation*}
d_{t}^{1}(G_{\mu }(f))(x)\geq t^{-n}\int_{C}|\Delta _{h}^{1}(G_{\mu
}(f))(x)|dh\geq ct^{\mu },
\end{equation*}%
which yields $G_{\mu }(f)$ does not belong to $B_{p,q}^{\frac{n+\alpha }{p}%
+\mu }(\mathbb{R}^{n},|\cdot |^{\alpha })$.

\textbf{Case 2.} $0<p\leq 1$. Let $p\leq 1<p_{1}$. The Sobolev embeddings
given in Theorem \ref{embeddings5} leads to 
\begin{equation*}
B_{p,v}^{\frac{n+\alpha }{p}+\mu }(\mathbb{R}^{n},|\cdot |^{\alpha
})\hookrightarrow B_{p_{1},v}^{\frac{n+\alpha }{p_{1}}+\mu }(\mathbb{R}%
^{n},|\cdot |^{\alpha }),\quad 0<\nu \leq \infty .
\end{equation*}%
Case 1\ yields the desired result. This completes the proof.
\end{proof}

\begin{remark}
We mention that we used only the function $f_{\mu ,0}$ of Lemma \ref%
{Si-funct1}, but $f_{\mu ,\delta }$, $\delta >0$ play an import role in the
problem of mapping properties of Nemytzkij operators in the classical Besov
and Triebel-Lizorkin spaces, see \cite{Runst-Sickel1996}.
\end{remark}

Before we prove the second result of this paper, we need another lemma. Let $%
\varrho $ be a $C^{\infty }$ function on $\mathbb{R}$ such that $\varrho
(x)=1$ for $x\leq e^{-3}$ and $\varrho (x)=0$ for $x\geq e^{-2}$. Let $%
(\lambda ,\sigma )\in \mathbb{R}^{2}$ and%
\begin{equation}
f_{\lambda ,\sigma }(x)=|\log |x||^{\lambda }|\log |\log |x|||^{-\sigma
}\varrho (|x|).  \label{triebel-function}
\end{equation}%
As in \cite{Bourdaud04} let $U_{q}$ be the set of $(\lambda ,\sigma )\in 
\mathbb{R}^{2}$ such that:

\textbullet\ $\lambda =1-\frac{1}{q}$ and $\sigma >\frac{1}{q}$, or $\lambda
<1-\frac{1}{q}$, in case $1<q<\infty $,

\textbullet\ $\lambda =0$ and $\sigma >0$, or $\lambda <0$, in case $q=1$,

\textbullet\ $\lambda =1$ and $\sigma \geq 0$, or $\lambda <1$, in case $%
q=\infty .$

\begin{lemma}
\label{Bourdaud-Triebel}Let $(\lambda ,\sigma )\in \mathbb{R}^{2},1\leq
p<\infty ,1\leq q\leq \infty ,\alpha >-n$ and 
\begin{equation}
(\lambda ,\sigma )\in U_{q}.  \label{condition1}
\end{equation}%
Let $f_{\lambda ,\sigma }$ be the function defined by %
\eqref{triebel-function}\textrm{. }We have $f_{\lambda ,\sigma }\in B_{p,q}^{%
\frac{n+\alpha }{p}}(\mathbb{R}^{n},|\cdot |^{\alpha })$. In the case $%
\alpha \geq 0$, the condition \eqref{condition1} becomes necessary.
\end{lemma}

\begin{proof}[\textbf{Proof of Theorem \protect\ref{Result2}}]
We decompose the proof into two steps.

\textbf{Step 1.} \textit{Proof of (i).}\textbf{\ }First assume that $s<\frac{%
n+\alpha }{p}$. Let $\tau $ be a positive number such that\ $(\frac{n+\alpha 
}{p}-s)\frac{1}{\mu }<\tau <\frac{n+\alpha }{p}-s$. Let $f(x)=\varrho
(x)|x|^{-\tau }$. From Lemma \ref{Si-funct1}, $f\in B_{p,q}^{s}(\mathbb{R}%
^{n},|\cdot |^{\alpha })\backslash L^{\infty }(\mathbb{R}^{n})$. But $%
|f(x)|^{\mu }=(\varrho (x))^{\mu }|x|^{-\tau \mu }\notin B_{p,q}^{s}(\mathbb{%
R}^{n},|\cdot |^{\alpha })$. Now assume that $s=\frac{n+\alpha }{p}$ and $%
q>1 $. Let $f$ be the function defined by $\mathrm{\eqref{triebel-function}}$
with $\lambda =1-\frac{1}{q}$ and $\sigma =\frac{1}{2}+\frac{1}{q}$. From
Lemma \ref{Bourdaud-Triebel} 
\begin{equation*}
f\in B_{p,q}^{\frac{n+\alpha }{p}}(\mathbb{R}^{n},|\cdot |^{\alpha
})\backslash L^{\infty }(\mathbb{R}^{n}).
\end{equation*}%
But $|f|^{\mu }\notin B_{p,q}^{\frac{n+\alpha }{p}}(\mathbb{R}^{n},|\cdot
|^{\alpha })$, since $\mu (1-\frac{1}{q})>1-\frac{1}{\beta }$.

\textbf{Step 2. }\textit{Proof of (ii).} Take $\delta >0$ be such that $s-%
\frac{n+\alpha }{p}<\delta <\frac{1}{\mu }(s-\frac{n+\alpha }{p})$ and $%
f(x)=\left\vert x\right\vert ^{\delta }\varrho \left( \left\vert
x\right\vert \right) $, $x\in \mathbb{R}^{n}$. From Lemma \ref{Si-funct1} we
easily obtain\ $f\in B_{p,q}^{s}(\mathbb{R}^{n},|\cdot |^{\alpha })\cap
L^{\infty }(\mathbb{R}^{n})$\ and $f^{\mu }\notin B_{p,q}^{s}(\mathbb{R}%
^{n},|\cdot |^{\alpha })$.$\newline
$The proof is complete.
\end{proof}

\begin{proof}[\textbf{Proof of Lemma \protect\ref{Bourdaud-Triebel}}]
For clarity, we split the proof into two steps.

\textbf{Step 1.} Sufficiency in part (i). From Theorem \ref{embeddings5}, 
\begin{equation*}
B_{1,q}^{n+\alpha }(\mathbb{R}^{n},|\cdot |^{\alpha })\hookrightarrow
B_{p,q}^{\frac{n+\alpha }{p}}(\mathbb{R}^{n},|\cdot |^{\alpha }),
\end{equation*}%
so we need only to prove that $f_{\lambda ,\sigma }\in B_{1,q}^{n+\alpha }(%
\mathbb{R}^{n},|\cdot |^{\alpha })$.\ Notice that 
\begin{equation*}
\big\|\mathcal{F}^{-1}\varphi _{j}\ast f\big\Vert_{L^{1}(\mathbb{R}%
^{n},|\cdot |^{\alpha })}<\infty ,\quad j=0,1.
\end{equation*}%
Indeed, the left-hand side is bounded by%
\begin{equation*}
\int_{|y|\leq e^{-2}}|f_{\lambda ,\sigma }(y)|\big\|\mathcal{F}^{-1}\varphi
_{j}(\cdot -y)\big\|_{L^{1}(\mathbb{R}^{n},|\cdot |^{\alpha })}dy,\quad
j=0,1.
\end{equation*}%
We write%
\begin{align*}
\big\|\mathcal{F}^{-1}\varphi _{j}(\cdot -y)\big\|_{L^{1}(\mathbb{R}%
^{n},|\cdot |^{\alpha })}=& c\sum_{k\in \mathbb{Z},2^{k-2}\leq
e^{-2}}2^{k\alpha }\big\|\mathcal{F}^{-1}\varphi _{j}(\cdot -y)\chi _{k}%
\big\|_{1} \\
& +c\sum_{k\in \mathbb{Z},2^{k-2}>e^{-2}}2^{k\alpha }\big\|\mathcal{F}%
^{-1}\varphi _{j}(\cdot -y)\chi _{k}\big\|_{1},
\end{align*}%
where $c>0$ is independent of $j$. The first sum clearly is bounded by%
\begin{equation*}
\sum_{k\in \mathbb{Z},2^{k-2}\leq e^{-2}}2^{k(\alpha +n)}<\infty ,
\end{equation*}%
since $\alpha +n>0$. The second sum can be estimated by%
\begin{equation*}
\sum_{k\in \mathbb{Z},2^{k-2}>e^{-2}}2^{k(\alpha +n-N)q}<\infty ,
\end{equation*}%
by taking $N>\alpha +n$. Therefore it suffices to prove the following:%
\begin{equation*}
\sum_{j=2}^{\infty }2^{j(n+\alpha )\beta }\big\|\mathcal{F}^{-1}\varphi
_{j}\ast f\big\Vert_{L^{1}(\mathbb{R}^{n},|\cdot |^{\alpha })}^{\beta
}<\infty ,
\end{equation*}%
where $\{\varphi _{j}\}_{j\in \mathbb{N}_{0}}$ is a partition of unity. From 
\cite[p. 272]{Bourdaud04},%
\begin{equation*}
|x|^{2v}|\mathcal{F}^{-1}\varphi _{j}\ast f_{\lambda ,\sigma }(x)|\lesssim
2^{-2jv}\varepsilon _{j},\quad x\in \mathbb{R}^{n},j\geq 2,v\in \mathbb{N}%
_{0},
\end{equation*}%
with%
\begin{equation*}
\varepsilon _{j}=j^{\lambda -1}(\log j)^{-\sigma }\text{\quad if\quad }%
\lambda \neq 0,\quad \varepsilon _{j}=j^{-1}(\log j)^{-\sigma -1}\text{\quad
if\quad }\lambda =0.
\end{equation*}%
Then we split%
\begin{equation*}
\sum_{k=-\infty }^{\infty }2^{k\alpha }\big\|(\mathcal{F}^{-1}\varphi
_{j}\ast f_{\lambda ,\sigma })\chi _{k}\big\|_{1}=I_{1,j}+I_{2,j},\text{%
\quad }j\geq 2,
\end{equation*}%
where%
\begin{equation*}
I_{1,j}=\sum_{k=-\infty }^{-j}2^{k\alpha }\big\|(\mathcal{F}^{-1}\varphi
_{j}\ast f_{\lambda ,\sigma })\chi _{k}\big\|_{1},\quad
I_{2,j}=\sum_{k=-j+1}^{\infty }2^{k\alpha }\big\|(\mathcal{F}^{-1}\varphi
_{j}\ast f_{\lambda ,\sigma })\chi _{k}\big\|_{1}.
\end{equation*}%
It is easily seen that $I_{1,j}\lesssim \varepsilon _{j}\sum_{k=-\infty
}^{-j}2^{k(\alpha +n)},j\geq 2$. Therefore%
\begin{equation*}
\sum_{j=2}^{\infty }2^{j(\alpha +n)q}(I_{1,j})^{q}\lesssim
\sum_{j=2}^{\infty }\varepsilon _{j}^{q}\sum_{k=-\infty
}^{-j}2^{(k+j)(\alpha +n)}\lesssim \sum_{j=2}^{\infty }\varepsilon
_{j}^{q}<\infty .
\end{equation*}%
Now%
\begin{equation*}
I_{2,j}\lesssim \varepsilon _{j}\sum_{k=-j+1}^{\infty }2^{(k\alpha -2jv)}%
\big\||\cdot |^{-2v}\chi _{k}\big\|_{1}\lesssim \varepsilon
_{j}\sum_{k=-j+1}^{\infty }2^{k(\alpha -2v+n)-2jv}
\end{equation*}%
for any $j\geq 2$. Hence%
\begin{equation*}
\sum_{j=2}^{\infty }2^{j(\alpha +n)q}(I_{2,j})^{q}\lesssim
\sum_{j=2}^{\infty }\varepsilon _{j}^{q}\sum_{k=-j+1}^{\infty
}2^{(k+j)(\alpha -2v+n)}\lesssim \sum_{j=2}^{\infty }\varepsilon
_{j}^{q}<\infty ,
\end{equation*}%
by taking $v>\frac{\alpha +n}{2}$.

\textbf{Step 2.} Let us assume $(\lambda ,\sigma )\notin U_{q}$ and $\alpha
\geq 0$. We are going to prove that $f_{\lambda ,\sigma }\notin B_{p,q}^{%
\frac{n+\alpha }{p}}(\mathbb{R}^{n},|\cdot |^{\alpha })$, but this follows
by the embeddings $B_{p,q}^{\frac{n+\alpha }{p}}(\mathbb{R}^{n},|\cdot
|^{\alpha })\hookrightarrow B_{p,q}^{\frac{n}{p}}$, and $f_{\lambda ,\sigma
}\in B_{p,q}^{\frac{n}{p}}$, $1\leq p,q\leq \infty $ if and only if $%
(\lambda ,\sigma )\in U_{q}$, see \cite[Proposition 2]{Bourdaud04}.
\end{proof}

In order to prove Theorem \ref{Result3 copy(1)}, we need the following
result.

\begin{proposition}
\label{Key-1}Let $\beta >0,1\leq p<\infty ,-n<\alpha <n(p-1),$%
\begin{equation*}
0<\max \Big(\delta +\frac{n}{p},\delta +\frac{n+\alpha }{p}\Big)<2(\beta
+1)\quad \text{and}\quad \sigma =\frac{\delta +\frac{n+\alpha }{p}}{\beta +1}%
.
\end{equation*}%
Let $g\in B_{\infty ,\infty }^{\gamma }(\mathbb{R})$ for some $\sigma
<\gamma $. The function 
\begin{equation*}
f(x)=\left\vert x\right\vert ^{\delta }g(\left\vert x\right\vert ^{-\beta
})\varrho \left( \left\vert x\right\vert \right)
\end{equation*}%
belongs to $B_{p,\infty }^{\sigma }(\mathbb{R}^{n},|\cdot |^{\alpha })${. }
\end{proposition}

Firstly, let us prove our last result.

\begin{proof}[\textbf{Proof of Theorem \protect\ref{Result3 copy(1)}}]
It is easy to see that $s<2$. Let $\beta \in \mathbb{R}$ be such that 
\begin{equation*}
\frac{1+\alpha }{sp}-1<\beta <\frac{2}{s}-1
\end{equation*}%
and%
\begin{equation*}
s(\beta +1)-\frac{1+\alpha }{p}<\delta <\min \Big(2-\frac{1+\alpha }{p},%
\frac{1}{\mu }\big(s(\beta +1)-\frac{1+\alpha }{p}\big)\Big).
\end{equation*}%
Since $s\geq \frac{\alpha }{p\mu }$, we obtain 
\begin{equation*}
\frac{\alpha \beta -1}{p\mu }\leq s(\beta +1)-\frac{1+\alpha }{p},
\end{equation*}%
which yeilds that 
\begin{equation}
\beta <\frac{\delta \mu p+1}{\alpha }.  \label{beta}
\end{equation}%
Let $f$ be as in Proposition \ref{Key-1} with $g(t)=(\sin ^{2}\frac{t}{2})^{%
\frac{1}{\mu }}\in B_{\infty ,\infty }^{\mu }(\mathbb{R})$. Then%
\begin{equation*}
f\in B_{p,\infty }^{\frac{\delta +\frac{1+\alpha }{p}}{\beta +1}}(\mathbb{R}%
,|\cdot |^{\alpha })\cap L^{\infty }(\mathbb{R})\hookrightarrow B_{p,q}^{s}(%
\mathbb{R},|\cdot |^{\alpha })\cap L^{\infty }(\mathbb{R}).
\end{equation*}%
Let $\beta _{1}>0$ be such that%
\begin{equation*}
\beta _{1}+1>\max \Big(\frac{(\beta +1)(\delta \mu p+1)}{\delta \mu
p+1-\beta \alpha },\delta \mu +\frac{1+\alpha }{p}\Big).
\end{equation*}%
This together with \eqref{beta} guarantees the embedding%
\begin{equation*}
B_{p,q}^{\frac{\delta \mu +\frac{1+\alpha }{p}}{\beta +1}}(\mathbb{R},|\cdot
|^{\alpha })\hookrightarrow B_{p,q}^{\frac{\delta \mu +\frac{1}{p}}{\beta
_{1}+1}},
\end{equation*}%
see Theorem \ref{embeddings5}. From \cite[Proposition 3]{Bourdaud04}, $%
f^{\mu }$ does not belong to $B_{p,q}^{\frac{\delta \mu +\frac{1}{p}}{\beta
_{1}+1}}$ and hence $f^{\mu }\notin B_{p,q}^{s}(\mathbb{R},|\cdot |^{\alpha
})${. }This completes the proof.
\end{proof}

\begin{proof}[\textbf{Proof of Proposition \protect\ref{Key-1}}]
Observe that $f\in L^{p}(\mathbb{R}^{n},|\cdot |^{\alpha })$. From Theorem %
\ref{means-diff-cha1}, we need to prove that%
\begin{equation}
\sup_{0<t\leq \frac{1}{2}e^{-2}}t^{-\sigma }\big\Vert d_{t}^{m}f\big\Vert%
_{L^{p}(\mathbb{R}^{n},|\cdot |^{\alpha })}<\infty ,\quad 0<\sigma <m\leq 2.
\label{whatweneed}
\end{equation}%
We will divide the proof into two steps.

\textbf{Step 1.}\ We\ will\ prove that $f\in B_{p,\infty }^{\sigma }(\mathbb{%
R}^{n},|\cdot |^{\alpha })$ with $0<\sigma <1$. We can only assume that $%
\gamma <1$. Let us estimate $\big\Vert d_{t}^{1}f\big\Vert_{L^{p}(\mathbb{R}%
^{n},|\cdot |^{\alpha })}$ for any $0<t\leq \frac{1}{2}e^{-2}$. Obviously, $%
d_{t}^{1}f(x)=0$ for any $x\in \mathbb{R}^{n}$ such that $\left\vert
x\right\vert \geq 2e^{-2}$ and $0<t\leq \frac{1}{2}e^{-2}$. We set $A=\{x\in 
\mathbb{R}^{n}:\left\vert x\right\vert <2e^{-2}\}$. By the equivalent
quasi-norm \eqref{discreteversion}, we see that \textit{\ }%
\begin{align}
\big\|(d_{t}^{1}f)\chi _{A}\big\|_{L^{p}(\mathbb{R}^{n},|\cdot |^{\alpha
})}^{p}\approx & \sum\limits_{k=-\infty }^{\infty }2^{k\alpha }\big\|%
(d_{t}^{1}f)\chi _{A\cap C_{k}}\big\|_{p}^{p}  \notag \\
=& H_{1}(t)+H_{2}(t),  \label{est-H}
\end{align}%
where 
\begin{equation*}
H_{1}(t)=c\sum\limits_{k\in \mathbb{Z},2^{k}<2t^{\frac{1}{\beta +1}%
}}2^{k\alpha }\big\|(d_{t}^{1}f)\chi _{A\cap C_{k}}\big\|_{p}^{p},\quad
H_{2}(t)=c\sum\limits_{k\in \mathbb{Z},2^{k}\geq 2t^{\frac{1}{\beta +1}%
}}2^{k\alpha }\big\|(d_{t}^{1}f)\chi _{A\cap C_{k}}\big\|_{p}^{p}.
\end{equation*}%
In what follows, we estimate each term on the right hand side of $\mathrm{%
\eqref{est-H}}$. To do this, note first 
\begin{align*}
\big\|(d_{t}^{1}f)\,\chi _{A\cap C_{k}}\big\|_{p}^{p}=& \int_{B(0,2t^{\frac{1%
}{\beta +1}})\cap A}\big(d_{t}^{1}f(x)\big)^{p}\chi _{k}(x)dx \\
& +\int_{(\mathbb{R}^{n}\backslash B(0,2t^{\frac{1}{\beta +1}}))\cap A}\big(%
d_{t}^{1}f(x)\big)^{p}\chi _{k}(x)dx \\
=& T_{1,k}(t)+T_{2,k}(t),\quad k\in \mathbb{Z}.
\end{align*}%
For clarity, we split this step into two substeps and conclusion.

\textbf{Substep 1.1. }\textit{Estimation of\ }$H_{1}$. Since $T_{2,k}(t)=0$
if $2^{k}<2t^{\frac{1}{\beta +1}},0<t\leq \frac{1}{2}e^{-2}\ $and $k\in 
\mathbb{Z}$, we need only to estimate $T_{1,k}(t)$. Let $x\in B(0,2t^{\frac{1%
}{\beta +1}})\cap A\cap C_{k}$. By Jensen's inequality%
\begin{equation*}
\big(d_{t}^{1}f(x)\big)^{p}\lesssim t^{-n}\int_{|h|<t}\left\vert f\left(
x+h\right) \right\vert ^{p}dh+\left\vert f\left( x\right) \right\vert ^{p}.
\end{equation*}%
Using the fact that 
\begin{equation*}
|x+h|\leq |x|+|h|<3t^{\frac{1}{\beta +1}}\text{,}\quad \left\vert
h\right\vert <t,
\end{equation*}%
it is easy to see that%
\begin{equation}
T_{1,k}(t)\lesssim t^{-n}\int_{|h|<t}\int_{B(-h,3t^{\frac{1}{\beta +1}})\cap
A\cap C_{k}}\left\vert f(x+h)\right\vert ^{p}dxdh+\int_{B(0,2t^{\frac{1}{%
\beta +1}})\cap A\cap C_{k}}\left\vert f(x)\right\vert ^{p}dx.
\label{est-H2}
\end{equation}%
Let 
\begin{equation*}
\max \Big(0,-\delta \frac{p}{n}\Big)<\frac{1}{\tau }<\min \Big(1,1+\frac{%
\alpha }{n}\Big).
\end{equation*}%
By H\"{o}lder's inequality and since $g,\varrho \in L^{\infty }(\mathbb{R})$%
, we obtain 
\begin{align}
\int_{B(-u,3t^{\frac{1}{\beta +1}})\cap A\cap C_{k}}\left\vert
f(x+u)\right\vert ^{p}dx\lesssim & 2^{k\frac{n}{\tau ^{\prime }}}\Big(%
\int_{B(-u,3t^{\frac{1}{\beta +1}})}\left\vert f(x+u)\right\vert ^{p\tau }dx%
\Big)^{\frac{1}{\tau }}  \notag \\
\lesssim & 2^{k\frac{n}{\tau ^{\prime }}}\Big(\int_{|z|<3t^{\frac{1}{\beta +1%
}}}\left\vert f(z)\right\vert ^{p\tau }dz\Big)^{\frac{1}{\tau }}  \notag \\
\lesssim & 2^{k\frac{n}{\tau ^{\prime }}}\Big(\int_{0}^{3t^{\frac{1}{\beta +1%
}}}r^{p\tau \delta +n-1}dr\Big)^{\frac{1}{\tau }}  \notag \\
\lesssim & \left\vert t\right\vert ^{\frac{\delta p+\frac{n}{\tau }}{\beta +1%
}}2^{k\frac{n}{\tau ^{\prime }}},  \label{est-H1}
\end{align}%
since $\delta +\frac{n}{p\tau }>0$, where $u\in \{0,h\}\ $and the implicit
constant is independent of $k$ and $t$. Plugging \eqref{est-H1}\ into %
\eqref{est-H2}, we obtain 
\begin{equation}
H_{1}(t)\leq ct^{\sigma p}\sum\limits_{k\in \mathbb{Z},2^{k}<2t^{\frac{1}{%
\beta +1}}}\Big(\frac{2^{k}}{t^{\frac{1}{\beta +1}}}\Big)^{(\alpha +\frac{n}{%
\tau ^{\prime }})}\leq ct^{\sigma p},  \label{key-est1}
\end{equation}%
since $\frac{n}{\tau ^{\prime }}+\alpha >0$, where $c>0$ is independent of $%
t $.

\textbf{Substep 1.2. }\textit{Estimation of\ }$H_{2}$. The situation is
quite different and more complicated. As in\ Substep 1.1 one finds that%
\begin{equation*}
T_{1,k}(t)\leq \int_{B(0,3t^{\frac{1}{\beta +1}})}\left\vert f\left(
x\right) \right\vert ^{p}\,dx\lesssim t^{\frac{\delta p+n}{\beta +1}},
\end{equation*}%
since $\delta +\frac{n}{p}>0$. Therefore%
\begin{equation*}
\sum\limits_{k\in \mathbb{Z},2^{k}\geq 2t^{\frac{1}{\beta +1}}}2^{k\alpha
}T_{1,k}(t)\leq ct^{(\delta +\frac{n}{p})\frac{p}{\beta +1}%
}\sum\limits_{k\in \mathbb{Z},2t^{\frac{1}{\beta +1}}\leq 2^{k}\leq 4t^{%
\frac{1}{\beta +1}}}2^{k\alpha }\leq ct^{\sigma p},
\end{equation*}%
where $c>0$ is independent of $t$.

\textit{Estimation of }$T_{2,k}(t)$. We decompose $\triangle _{h}^{1}f$\
into three parts%
\begin{equation*}
\triangle _{h}^{1}f(x)=\omega _{1}(x,h)+\omega _{2}(x,h)+\omega _{3}(x,h),
\end{equation*}%
where%
\begin{equation*}
\omega _{1}(x,h)=\left\vert x\right\vert ^{\delta }\big(g(\left\vert
x+h\right\vert ^{-\beta })-g(\left\vert x\right\vert ^{-\beta })\big)\varrho
\left( \left\vert x+h\right\vert \right) ,
\end{equation*}%
\begin{equation*}
\omega _{2}(x,h)=\big(|x+h|^{\delta }-|x|^{\delta })g(\left\vert
x+h\right\vert ^{-\beta }\big)\varrho (\left\vert x+h\right\vert )
\end{equation*}%
and%
\begin{equation*}
\omega _{3}(x,h)=\left\vert x\right\vert ^{\delta }g(\left\vert x\right\vert
^{-\beta })\big(\varrho (\left\vert x+h\right\vert )-\varrho (\left\vert
x\right\vert )\big).
\end{equation*}%
Define%
\begin{equation*}
\tilde{\omega}_{i}(x,t)=t^{-n}\int_{|h|<t}|\omega _{i}(x,h)|dh,\quad i\in
\{1,2,3\}.
\end{equation*}%
By the mean value theorem we have%
\begin{equation*}
\big||x+h|^{-\beta }-|x|^{-\beta }\big|\leq c|h||x|^{-\beta -1},\quad
\left\vert x\right\vert \geq 2\left\vert h\right\vert ^{\frac{1}{\beta +1}},
\end{equation*}%
which together with the fact that $g\in B_{\infty ,\infty }^{\gamma }$ we
obtain that%
\begin{equation*}
\big|g(|x+h|^{-\beta })-g(|x|^{-\beta })\big|\leq c|h|^{\gamma }\big\|g\big\|%
_{B_{\infty ,\infty }^{\gamma }}|x|^{-\gamma (\beta +1)},
\end{equation*}%
where $c>0$ is independent of $h$. Therefore,%
\begin{equation*}
\tilde{\omega}_{1}(x,t)\lesssim t^{\gamma }\big\|g\big\|_{B_{\infty ,\infty
}^{\gamma }}|x|^{-\gamma (\beta +1)},
\end{equation*}%
which yields%
\begin{align*}
\int_{\{x:2t^{\frac{1}{\beta +1}}\leq |x|\leq 2e^{-2}\}\cap C_{k}}|\tilde{%
\omega}_{1}(x,t)|^{p}dx\lesssim & t^{\gamma p}\int_{C_{k}}|x|^{\delta
p-\gamma (\beta +1)p}dx \\
\lesssim & t^{\gamma p}2^{k(\delta p-\gamma (\beta +1)p+n}.
\end{align*}%
Consequently%
\begin{align}
\sum\limits_{k\in \mathbb{Z},2^{k}\geq 2t^{\frac{1}{\beta +1}}}2^{k\alpha }%
\big\|\tilde{\omega}_{1}(\cdot ,t)\chi _{k}\big\|_{p}^{p}\lesssim &
t^{\gamma p}\sum\limits_{k\in \mathbb{Z},2^{k}\geq 2t^{\frac{1}{\beta +1}%
}}2^{k(\delta -\gamma (\beta +1)+\frac{n+\alpha }{p})p}  \notag \\
\lesssim & t^{\sigma p}\sum\limits_{k\in \mathbb{Z},2^{k}\geq 2t^{\frac{1}{%
\beta +1}}}\Big(\frac{2}{t^{\frac{1}{\beta +1}}}\Big)^{(\sigma -\gamma
)(\beta +1)p}  \notag \\
\lesssim & t^{\sigma p},  \label{omega1}
\end{align}%
since $\sigma <\gamma $. We have%
\begin{equation}
\big|\left\vert x+h\right\vert ^{\delta }-\left\vert x\right\vert ^{\delta }%
\big|\leq c|h||x+\theta h|^{\delta -1},\quad 0<\theta <1,  \label{est-omega2}
\end{equation}%
becuase of $\left\vert x\right\vert \geq 2t^{\frac{1}{\beta +1}}>2\left\vert
h\right\vert ^{\frac{1}{\beta +1}}$, where the positive constant $c$ is
independent of $h$ and $t$. From 
\begin{equation*}
\frac{1}{2}|x|\leq |x+\theta h|\leq \frac{3}{2}|x|,\quad g,\varrho \in
L^{\infty }(\mathbb{R})
\end{equation*}%
and \eqref{est-omega2} we immediately deduce that 
\begin{equation*}
\sum\limits_{k\in \mathbb{Z},2^{k}\geq 2t^{\frac{1}{\beta +1}}}2^{k\alpha }%
\big\|\tilde{\omega}_{2}(\cdot ,t)\,\chi _{\{x:|x|\leq 2e^{-2}\}\cap C_{k}}%
\big\|_{p}^{p}\lesssim t^{p}\sum\limits_{k\in \mathbb{Z},2t^{\frac{1}{\beta
+1}}\leq 2^{k}\leq 4e^{-2}}2^{k(\delta -1+\frac{n+\alpha }{p})p}
\end{equation*}%
which is bounded by%
\begin{equation}
S(t)=c\left\{ 
\begin{array}{ccc}
t^{(1-\frac{1}{\beta +1}+\sigma )p} & \text{if} & \delta -1+\frac{n+\alpha }{%
p}<0, \\ 
t^{p} & \text{if} & \delta -1+\frac{n+\alpha }{p}>0, \\ 
t^{p}\log \frac{1}{t} & \text{if} & \delta -1+\frac{n+\alpha }{p}=0,%
\end{array}%
\right.  \label{omega2}
\end{equation}%
for sufficiently small $t>0$. Obviously, 
\begin{equation}
\sum\limits_{k\in \mathbb{Z},2^{k}\geq 2t^{\frac{1}{\beta +1}}}2^{k\alpha }%
\big\|\tilde{\omega}_{3}(\cdot ,t)\,\chi _{\{x:|x|\leq 2e^{-2}\}\cap C_{k}}%
\big\|_{p}^{p}\lesssim t^{p}\sum\limits_{k\in \mathbb{Z},2^{k}\leq
4e^{-2}}2^{k(\delta +\frac{n+\alpha }{p})p}\lesssim t^{p}.  \label{omega3}
\end{equation}%
Collecting estimations \eqref{omega1}, \eqref{omega2} and \eqref{omega3}, we
derive 
\begin{equation}
H_{2}(t)\lesssim t^{\sigma p}+S(t).  \label{key-est2}
\end{equation}

\textbf{Conclusion.} By combining two estimates \eqref{key-est1} and %
\eqref{key-est2} we obtain $f\in B_{p,\infty }^{\sigma }(\mathbb{R}%
^{n},|\cdot |^{\alpha })$ but with 
\begin{equation*}
0<\sigma <1\text{\quad and\quad }\delta -1+\frac{n+\alpha }{p}\neq 0.
\end{equation*}%
Next we consider the case\ $\delta =1-\frac{n+\alpha }{p}$. Let $1\leq
p_{1}<p<p_{2}<\infty $ be such that $\frac{1}{p_{1}}<\frac{n}{n+\alpha }$ and%
\begin{equation*}
0<\max \Big(\delta +\frac{n}{p_{1}},\delta +\frac{n+\alpha }{p_{1}}\Big)%
<2(\beta +1).
\end{equation*}%
We set%
\begin{equation*}
\sigma _{i}=\frac{\delta +\frac{n+\alpha }{p_{i}}}{\beta +1},\quad i\in
\{0,1\},\quad \frac{1}{p}=\frac{\theta }{p_{1}}+\frac{1-\theta }{p_{2}}%
,\quad 0<\theta <1.
\end{equation*}%
Observe that $\delta -1+\frac{n+\alpha }{p_{1}}>0$ and $\delta -1+\frac{%
n+\alpha }{p_{2}}<0$, which yield that $f\in B_{p_{i},\infty }^{\sigma _{i}}(%
\mathbb{R}^{n},|\cdot |^{\alpha })$, $i\in \{0,1\}$. By H\"{o}lder's
inequality, we obtain%
\begin{equation}
\big\|f\big\|_{B_{p,\infty }^{\sigma }(\mathbb{R}^{n},|\cdot |^{\alpha
})}\leq \big\|f\big\|_{B_{p_{1},\infty }^{\sigma _{1}}(\mathbb{R}^{n},|\cdot
|^{\alpha })}^{\theta }\big\|f\big\|_{B_{p_{2},\infty }^{\sigma _{2}}(%
\mathbb{R}^{n},|\cdot |^{\alpha })}^{1-\theta }.  \label{interpolation}
\end{equation}%
This ensures that $f\in B_{p,\infty }^{\sigma }(\mathbb{R}^{n},|\cdot
|^{\alpha })$ but for $p>1$. Now assume that $p=1$. Let $-n<\alpha
_{1}<\alpha <\alpha _{2}<0$ be such that 
\begin{equation*}
0<\max (\delta +n,\delta +n+\alpha _{2})<2(\beta +1).
\end{equation*}%
We put%
\begin{equation*}
\sigma _{i}=\frac{\delta +n+\alpha _{i}}{\beta +1},\quad i\in \{0,1\},
\end{equation*}%
which yield that $f\in B_{1,\infty }^{\sigma _{i}}(\mathbb{R}^{n},|\cdot
|^{\alpha })$, $i\in \{0,1\}$. An interpolation inequality as in %
\eqref{interpolation} gives that $f\in B_{1,\infty }^{\sigma }(\mathbb{R}%
^{n},|\cdot |^{\alpha }),0<\sigma <1$.

\textbf{Step 2.}\textit{\ }In this step we prove that $f$ belongs to $f\in
B_{p,\infty }^{\sigma }(\mathbb{R}^{n},|\cdot |^{\alpha })$ with $1\leq
\sigma <2$. We can only assume that $\sigma <\gamma <2$. Then we split%
\begin{equation*}
\big\|(d_{t}^{2}f)\chi _{A}\big\|_{L^{p}(\mathbb{R}^{n},|\cdot |^{\alpha
})}^{p}=I_{1}+I_{2},
\end{equation*}%
where 
\begin{equation*}
I_{1}(t)=\sum\limits_{k\in \mathbb{Z},2^{k}<2t^{\frac{1}{\beta +1}%
}}2^{k\alpha }\big\|(d_{t}^{2}f)\chi _{A\cap C_{k}}\big\|_{p}^{p},\quad
I_{2}(t)=\sum\limits_{k\in \mathbb{Z},2^{k}\geq 2t^{\frac{1}{\beta +1}%
}}2^{k\alpha }\big\|(d_{t}^{2}f)\chi _{A\cap C_{k}}\big\|_{p}^{p}.
\end{equation*}%
We use the following decomposition: 
\begin{align*}
&\big\|(d_{t}^{2}f\,)\chi _{A\cap C_{k}}\big\|_{p}^{p} \\
=&\int_{B(0,2t^{\frac{1}{\beta +1}})\cap A}\left\vert d_{t}^{2}f\left(
x\right) \right\vert ^{p}\chi _{k}(x)dx\,+\int_{(\mathbb{R}^{n}\backslash
B(0,2t^{\frac{1}{\beta +1}}))\cap A}\left\vert d_{t}^{2}f\left( x\right)
\right\vert ^{p}\chi _{k}(x)dx \\
=&V_{1,k}(t)+V_{2,k}(t),\quad 0<t<1,k\in \mathbb{Z}.
\end{align*}%
We will divide the proof into two Substeps 2.1 and 2.2.

\textbf{Substep 2.1.}$\ $\textit{Estimation of }$I_{1}$. Obviously $%
V_{2,k}(t)=0$ if $2^{k}<2t^{\frac{1}{\beta +1}}$ and $k\in \mathbb{Z}$. We
have%
\begin{equation*}
\triangle _{h}^{2}f\left( x\right) =f\left( x+2h\right) +f\left( x\right)
-2f\left( x+h\right)
\end{equation*}%
and 
\begin{equation*}
|x+2h|\leq |x|+2|h|<4t^{\frac{1}{\beta +1}},
\end{equation*}%
if $x\in B(0,2t^{\frac{1}{\beta +1}})$ and $\left\vert h\right\vert <t$. In
this case, we use an argument similar to that used in Step 1 we find $%
I_{1}(t)\lesssim t^{\sigma p}.$

\textbf{Substep 2.2.}$\ $\textit{Estimation of}\textbf{\ }$I_{2}$. Using the
same type of arguments as in Step 1 it is easy to see that $V_{1,k}(t)\leq
ct^{\frac{\delta p+n}{\beta +1}}$, where $c>0$ is independent of $k$ and $t$
and 
\begin{equation*}
\sum\limits_{k\in \mathbb{Z},2^{k}\geq 2t^{\frac{1}{\beta +1}}}2^{k\alpha
}V_{1,k}(t)\leq ct^{\sigma p}.
\end{equation*}%
We decompose $\triangle _{h}^{2}f(x)$ into $\sum_{i=1}^{5}\varpi _{i}(x,h)$,
where%
\begin{equation*}
\varpi _{1}(x,h)=\left\vert x+h\right\vert ^{\delta }\big(g(\left\vert
x+2h\right\vert ^{-\beta })+g(\left\vert x\right\vert ^{-\beta
})-2g(\left\vert x+h\right\vert ^{-\beta })\big)\varrho \left( \left\vert
x+2h\right\vert \right) ,
\end{equation*}%
\begin{align*}
\varpi _{2}(x,h) =&(\left\vert x+2h\right\vert ^{\delta }-\left\vert
x+h\right\vert ^{\delta })g(\left\vert x+2h\right\vert ^{-\beta })\varrho
(\left\vert x+2h\right\vert ), \\
\varpi _{3}(x,h) =&(\left\vert x\right\vert ^{\delta }-\left\vert
x+h\right\vert ^{\delta })g(\left\vert x\right\vert ^{-\beta })\varrho
(\left\vert x\right\vert ),
\end{align*}%
\begin{equation*}
\varpi _{4}(x,h)=2\left\vert x+h\right\vert ^{\delta }g(\left\vert
x+h\right\vert ^{-\beta })\big(\varrho (\left\vert x+2h\right\vert )-\varrho
(\left\vert x+h\right\vert )\big)
\end{equation*}%
and%
\begin{equation*}
\varpi _{5}(x,h)=\left\vert x+h\right\vert ^{\delta }g(\left\vert
x\right\vert ^{-\beta })\big(\varrho (\left\vert x\right\vert )-\varrho
(\left\vert x+2h\right\vert )\big).
\end{equation*}%
Obviously we need only to estimate $\varpi _{1}$. We split%
\begin{equation*}
2g(\left\vert x+2h\right\vert ^{-\beta })+2g(\left\vert x\right\vert
^{-\beta })-4g(\left\vert x+h\right\vert ^{-\beta })
\end{equation*}%
into three terms i.e., $\vartheta _{1}(x,h)+\vartheta _{2}(x,h)+\vartheta
_{3}(x,h)$, where%
\begin{align*}
&\vartheta _{1}(x,h) \\
=&g(\left\vert x+2h\right\vert ^{-\beta })-g(2\left\vert x+h\right\vert
^{-\beta }-\left\vert x\right\vert ^{-\beta })+g(\left\vert x\right\vert
^{-\beta })-g(2\left\vert x+h\right\vert ^{-\beta }-\left\vert
x+2h\right\vert ^{-\beta }),
\end{align*}%
\begin{equation*}
\vartheta _{2}(x,h)=g(\left\vert x+2h\right\vert ^{-\beta })+g(2\left\vert
x+h\right\vert ^{-\beta }-\left\vert x+2h\right\vert ^{-\beta
})-2g(\left\vert x+h\right\vert ^{-\beta })
\end{equation*}%
and%
\begin{equation*}
\vartheta _{3}(x,h)=g(\left\vert x\right\vert ^{-\beta })+g(2\left\vert
x+h\right\vert ^{-\beta }-\left\vert x\right\vert ^{-\beta })-2g(\left\vert
x+h\right\vert ^{-\beta }).
\end{equation*}%
Define%
\begin{equation*}
\tilde{\vartheta}_{i}(x,t)=t^{-n}\int_{|h|<t}|\vartheta _{i}(x,h)|dh,\quad
i\in \{1,2,3\}.
\end{equation*}%
Let%
\begin{equation*}
(J_{i,k}(t))^{p}=\int_{(\mathbb{R}^{n}\backslash B(0,2t^{\frac{1}{\beta +1}%
}))\cap A}\left\vert x\right\vert ^{\delta p}|\tilde{\vartheta}%
_{i}(x,t)|^{p}\chi _{k}(x)dx,\quad i\in \{1,2,3\}.
\end{equation*}%
Observe that $g^{(1)}\in B_{\infty ,\infty }^{\gamma -1}(\mathbb{R}%
)\hookrightarrow L^{\infty }(\mathbb{R})$. Again by the mean value theorem;%
\begin{equation*}
\left\vert \left\vert x+2h\right\vert ^{-\beta }+\left\vert x\right\vert
^{-\beta }-2\left\vert x+h\right\vert ^{-\beta }\right\vert \leq
c|h|^{2}|x|^{-\beta -2},\quad \left\vert x\right\vert \geq 2\left\vert
h\right\vert ^{\frac{1}{\beta +1}},
\end{equation*}%
which yields that%
\begin{equation*}
(J_{1,k}(t))^{p}\lesssim t^{2p}\int_{A\cap C_{k}}|x|^{(\delta -(\beta
+2))p}\,dx\lesssim t^{2p}2^{k((\delta -(\beta +2))p+n)}.
\end{equation*}%
We also obtain%
\begin{equation*}
(J_{i,k}(t))^{p}\lesssim t^{\gamma p}\int_{A\cap C_{k}}|x|^{(\delta -\gamma
(\beta +1))p}\,dx\lesssim t^{\gamma p}2^{k((\delta -\gamma (\beta
+1))p+n)},\quad i\in \{2,3\}.
\end{equation*}%
Therefore%
\begin{align*}
\sum\limits_{k\in \mathbb{Z},2t^{\frac{1}{\beta +1}}\leq 2^{k}\leq
4e^{-2}}(J_{1,k}(t))^{p} \lesssim &t^{2p}\sum\limits_{k\in \mathbb{Z},2t^{%
\frac{1}{\beta +1}}\leq 2^{k}\leq 4e^{-2}}2^{k(\delta +\frac{n+\alpha }{p}%
-(\beta +2))p} \\
\lesssim &\max \big(t^{(2+\sigma -\frac{\beta +2}{\beta +1})p},t^{2p}\big)
\end{align*}%
and%
\begin{align*}
\sum\limits_{k\in \mathbb{Z},2^{k}\geq 2t^{\frac{1}{\beta +1}%
}}(J_{i,k}(h))^{p} \lesssim &t^{\gamma p}\sum\limits_{k\in \mathbb{Z}%
,2^{k}\geq 2t^{\frac{1}{\beta +1}}}2^{k(\delta +\frac{n+\alpha }{p}-\gamma
(\beta +1))p},\quad i\in \{2,3\} \\
\lesssim &t^{\sigma q},
\end{align*}%
since $\sigma <\gamma $. Hence%
\begin{equation*}
I_{2}(t)\lesssim t^{\sigma p}+\max \big(t^{(2+\sigma -\frac{\beta +2}{\beta
+1})p},t^{2p}\big).
\end{equation*}%
Collecting the estimates of $I_{1}$ and $I_{2}$ we have proved $f\in
B_{p,\infty }^{\sigma }(\mathbb{R}^{n},|\cdot |^{\alpha })$ with $1\leq
\sigma <2$. The proof is complete.
\end{proof}

\textbf{Acknowledgements}$\newline
$ The author would like to thank W. Sickel for valuable discussions and
suggestions. This work was supported\ by the General Direction of Higher
Education and Training under\ Grant No. C00L03UN280120220004, Algeria.


\begin{thebibliography}{99}
\bibitem{Bourdaud-Meyer91} Bourdaud, G., Meyer, Y.: Fonctions qui operent
sur les espaces de Sobolev, J. Funct. Anal, \textbf{97}, 351-360 (1991)

\bibitem{Bo91} Bourdaud, G.: Le calcul fonctionnel dans les espaces de
Sobolev, lnvent. Math. \textbf{104, }435-446\ (1991)

\bibitem{Bourdaud93} Bourdaud, G.: The functional calculus in Sobolev
spaces. In: Function Spaces, Differential Operators and Nonlinear Analysis
(ed. by Schmeisser, H.-J., Triebel, H.). Teubner-Texte Math. 133, pp.
127-142, Teubner, Stuttgart, Leipzig 1993

\bibitem{BK} Bourdaud, G., Kateb, D.: Fonctions qui op\`{e}rent sur certains
espaces de Besov, Ann. Inst. Fourier. \textbf{40}, 153-162 (1990)

\bibitem{Bourdaud04} Bourdaud, G.: A sharpness result for powers of Besov
functions, J. Funct. Spaces Appl, \textbf{2}(3), 267-277 (2004)

\bibitem{BCS06} Bourdaud, G., Lanza de Cristoforis, M., Sickel, W.:
Superposition operators and functions of bounded $p$-variation, Revista Mat.
Iberoamericana. 22, 455-487 (2006)

\bibitem{BMS10} Bourdaud, G., Moussai, M., Sickel, W.: Composition operators
in Lizorkin-Triebel spaces, J. Funct. Anal. \textbf{259}, 1098-1128 (2010).

\bibitem{Bui82} Bui, H.Q.: Weighted Besov and Triebel spaces: interpolation
by the real method, Hiroshima Math. J, \textbf{12}, 581--605 (1982)

\bibitem{CW90} Cazenave, T., Weissler, W. B.: The Cauchy problem for the
critical non linear Schrodinger equation in $H^{s}$, Nonlinear Analysis 
\textbf{14}, 807-836 (1990)

\bibitem{CFZ11} Cazenave, T., Fang, D., Zheng, H.: Continuous dependence for
NLS in fractional order spaces, Ann. I. H. Poincar\'{e}-AN \textbf{28},
135-147 (2011)

\bibitem{Da79} Dahlberg, B. J.: A note on Sobolev spaces, Proc. Symp. Pure
Math, \textbf{35 }(1), 183-85 (1979)

\bibitem{Drihem2013a} Drihem, D.: Embeddings properties on Herz-type Besov
and Triebel-Lizorkin spaces, Math. Ineq. and Appl, \textbf{16}(2), 439-460
(2013)

\bibitem{Drihem20} Drihem, D.: Caffarelli-Kohn-Nirenberg inequalities on
Besov and Triebel-Lizorkin-type spaces, arXiv:1808.08227

\bibitem{Dr21} Drihem, D.: Nemytzkij operators on Sobolev spaces with power
weights: I, J. Math. Sci (2022). https://doi.org/10.1007/s10958-022-05895-9

\bibitem{DrBanach} Drihem, D.: Composition operators on Herz-type
Triebel-Lizorkin spaces with application to semilinear parabolic equations,
Banach J. Math. Anal. \textbf{16}, 29 (2022)

\bibitem{Dr-Nim} Drihem, D.: Nemytzkij operators on Sobolev spaces with
power weights: II, preprint.

\bibitem{F98} Ribaud, F.: Cauchy problem for semilinear parabolic equations
with initial data in $H_{p}^{s}(\mathbb{R}^{n})$\ spaces,\textit{\ }Rev.
Mat. Iberoamericana \textbf{14}(1), 1-46\ (1998)

\bibitem{K03} Kateb, D.: On the boundedness of the mapping $f\mapsto
|f|^{\mu },%
\mu
>1$, on Besov Spaces, Math. Nachr, \textbf{248-249}, 110-128 (2003)

\bibitem{MM12} Meyries, M., Veraar, M.C.: Sharp embedding results for spaces
of smooth functions with power weights, Studia. Math. \textbf{208}(3)),
257-293 (2012)

\bibitem{Os91} Oswald, P.: On the boundedness of the mapping $f\rightarrow
|f|$ in Besov spaces, Comment. Univ. Carolinae, \textbf{33}, 57-66 (1992)

\bibitem{Runst-Sickel1996} Runst, T., Sickel, W.: Sobolev spaces of
fractional order, Nemytskij operators, and nonlinear partial differential
equations, de Gruyter, Berlin, 1996

\bibitem{Triebel1983} Triebel, H.: Theory of function spaces. Birkh\"{a}%
user, Basel (1983)

\bibitem{Xu05} Xu, J.: Equivalent norms of Herz-type Besov and
Triebel-Lizorkin spaces,\ J. Funct. Spaces Appl, \textbf{3}, 17-31 (2005)

$\newline
$
\end{thebibliography}

Douadi Drihem

M'sila University, Department of Mathematics,

Laboratory of Functional Analysis and Geometry of Spaces,

M'sila 28000, Algeria,

e-mail: \texttt{douadidr@yahoo.fr, douadi.drihem@univ-msila.dz}

\end{document}